\documentstyle[12pt, amssymb,amsthm]{article}

\DeclareSymbolFont{rsfs}{U}{rsfs}{m}{n}
\DeclareSymbolFontAlphabet{\mathscr}{rsfs}

\def\dd{{\mathfrak d}}

\def\gg{{\mathfrak g}} 
\def\hh{{\mathfrak h}} 
\def\ii{{\mathfrak i}}

\def\rr{{\mathfrak r}} 
 
\def\tt{{\mathfrak t}}

\def\sl{{\mathfrak sl}}

\def\VV{{\mathfrak V}}
\def\WW{{\mathfrak W}}

\def\Vir{{\VV \ii \rr}}

\def\CCC{{\mathbb C}}

\def\NNN{{\mathbb N}}

\def\SSS{{\mathbb S}}

\def\ZZZ{{\mathbb Z}}

\def\F{{\mathcal F}}

\def\t1{{\frac{1}{3}}}
\def\1/2{{\frac{1}{2}}}
\def\3/2{{\frac{3}{2}}}

\def\Ž{{\' e}}
\def\{{\` e}}
\def\{{\^ e}}
\def\ˆ{{\` a}}

 \newtheorem{lemma}{Lemma}[section]
 \newtheorem{proposition}[lemma]{Proposition}
  \newtheorem{examples}[lemma]{Examples}
   
  \newtheorem{theorem}[lemma]{Theorem}
\newtheorem{definition}[lemma]{Definition}

 \newtheorem{corollary}[lemma]{Corollary}
 \newtheorem{remark}[lemma]{Remark}

   \newtheorem{reminder}[lemma]{Reminder}  
    
      \newtheorem{convention}[lemma]{Convention}

\title{Neveu-Schwarz and operators algebras II   \\ \textit{Unitary series and characters}  }

\author{S\'ebastien Palcoux }

 \date{}   
                  
\begin{document}
\maketitle

\begin{abstract}
This paper is the second of a series giving a self-contained way from the Neveu-Schwarz algebra  to a new series of irreducible subfactors.
Here we give a  unitary complete proof of the classification of the unitary series of the Neveu-Schwarz algebra, by the way of GKO construction, Kac determinant and FQS criterion. We then obtain the characters directly, without Feigin-Fuchs resolutions.

\end{abstract}

\tableofcontents   \newpage
\section{Introduction}

\subsection{Background of the series}
In the $90$'s,  V. Jones and A. Wassermann started a program whose goal is to understand the unitary conformal field theory from the point of view of operator algebras (see \cite{2c}, \cite{2b}).  In \cite{2}, Wassermann defines and computes the Connes fusion of the irreducible positive energy representations of the loop group $LSU(n)$ at fixed level $\ell$, using primary fields, and with consequences in the theory of subfactors.  In  \cite{vtl} V. Toledano Laredo proves the Connes fusion rules for $LSpin(2n)$ using similar methods. Now, let Diff$(\SSS^{1})$ be the diffeomorphism group on the circle, its Lie algebra is the Witt algebra $\WW$ generated by $d_{n}$ ($n \in \ZZZ$), with $[d_{m} , d_{n}] = (m-n)d_{m+n}$. It admits a unique central extension called the Virasoro algebra $\Vir$. Its unitary positive energy representation theory and the character formulas can be deduced by a so-called Goddard-Kent-Olive (GKO) coset construction from the theory of $LSU(2)$ and the Kac-Weyl formulas (see \cite{1}, \cite{3a}). In \cite{loke}, T. Loke uses the coset construction to compute the Connes fusion for $\Vir$. 
Now, the Witt algebra admits two supersymmetric extensions $\WW_{0}$ and $\WW_{1/2}$ with central extensions called the Ramond and the Neveu-Schwarz algebras, noted $\Vir_{0}$ and $\Vir_{1/2}$.  In this series (\cite{NSOAI}, this paper and \cite{NSOAIII}), we naturally introduce $\Vir_{1/2}$ in the vertex superalgebra context of $L\sl_{2}$, we give a complete proof of the classification of its unitary positive energy representations, we obtain directly their character; then we give the Connes fusion rules, and an irreducible finite depth type II$_{1}$ subfactors for each representation of the discrete series.
Note that we could do the same for the Ramond algebra $\Vir_{0} $, using twisted vertex module over the vertex operator algebra of the Neveu-Schwarz algebra $\Vir_{1/2}$, as  R. W. Verrill  \cite{verrill} and  Wassermann \cite{wass2} do for twisted loop groups.

\subsection{Overview of the paper}
 Let  $\gg = \sl_{2}$, using  theta functions framework, we obtain the decomposition of $H =  \F_{NS}^{\gg} \otimes ( L(j,\ell) \otimes \F_{NS}^{\gg})$ as $\widehat{\gg}$-module. The multiplicity spaces of  irreducible components $H_{k}$ are  superintertwiners space $Hom_{\widehat{\gg}}(H_{k}, H)$; we deduce their character as module of $ \WW_{1/2}$, which acts on with  $L(c_{m},h_{pq}^{m})$ as submodule by GKO construction. The unitarity of the discrete series follows.  \\ 
$\begin{array}{c} \end{array}$ \hspace{0,25cm}   We define irreducible polynomial $\varphi_{pq}(c,h)$ from $(c_{m}, h_{pq}^{m})$. The Kac  determinant $det_{n}(c,h)$ of the sesquilinear form on $V(c,h)$ at level $n$ is easily interpolate, as a product of $\varphi_{pq}$,  computing the first examples. To prove it,  we enlight links between previous characters results and singular vectors $s$ \\ (i.e. $G_{1/2}.s = G_{3/2}.s = 0$),   whose the existence vanishes  $det_{n}$. 

  A negative Kac determinant shows easily a ghost on the region between the curves $h=h_{pq}^{c}$. Now, we go from the no-ghost region $h > 0$, $c >3/2 $ to  an order $1$ vanishing curve $C$; then, on the other side, there is a ghost. By transversality, it pass on the curve intersecting $C$ next. And so on each curves,  excepting `first intersections':  discrete series. Theorem \ref{satya} follows.
  
  Finally, a coherence argument between the characters of the multiplicity spaces $M_{pq}^{m}$ and its irreducibles (on discrete series by FQS),  shows $M_{pq}^{m}$ without others irreducibles that $L(c_{m} , h_{rs}^{m})$. So, $M_{pq}^{m} =  L(c_{m},h_{p,q}^{m})$ and we obtain the character of $L(c_{m},h_{p,q}^{m})$ as the character of $M_{pq}^{m}$, ever known by GKO construction. Theorem \ref{sati} follows.

\subsection{Main results}
The irreducible positive energy representations of the Neveu-Schwarz algebra $\Vir_{1/2}$ are denoted $L(c,h)$ with $\Omega$ its cyclic vector. Our purpose  is to give a complete proof of the classification of unitary representations, in such a way that we obtain directly the characters of the discrete series, without Feigin-Fuchs resolution  \cite{14}. The Neveu-Schwarz algebra is defined by:
 \begin{center}  $ \left\{  \begin{array}{l} 
\lbrack L_{m},L_{n} \rbrack \hspace{0,3cm} = (m-n)L_{m+n} +\frac{C}{12}(m^{3} - m) \delta_{m+n}  Ê \\
\lbrack G_{r},L_{n} \rbrack \hspace{0,28cm} = (m-\frac{n}{2})G_{r+n}    \\
\lbrack G_{r},G_{s}\rbrack_{+} = 2L_{r+s} + \frac{C}{3}(r^{2}- \frac{1}{4})\delta_{r+s} 
\end{array}   \right.  $ \end{center}
 with $m$, $n \in \ZZZ$,  $r$, $s \in \ZZZ + \1/2 $,    $L_{n}^{\star} = L_{-n} $,  $G_{r}^{\star} = G_{-r} $. \\ 
 Positive energy means that $L(c,h)  = H =Ê\bigoplus H_{n}$, with $n \in \1/2 \NNN$, such that $L_{0}\xi = (n+h)\xi$ on $H_{n}$ and $H_{0} = \CCC \Omega $ (with $C\Omega = c \Omega$). 
 \begin{lemma}  If  $L(c,h)$ is unitary,   then $ c,h \ge 0$     \end{lemma}
\begin{theorem} \label{satya} The classification of  unitary representations $L(c,h)$ is:  
\begin{description}
\item[(a)] Continuous series:   $c  \ge 3/2$ and  $h \ge 0$.     
 \item[(b)] Discrete series: $(c,h) = (c_{m},h_{pq}^{m} )$    with:
 \begin{displaymath} c_{m} =  \frac{3}{2} (1 - \frac{ 8}{m(m+2) }) \quad \textrm{and} \quad   h_{pq}^{m}  =  \frac{((m+2)p-mq)^{2}-4}{8m(m+2)}  \end{displaymath}
with integers  $m \ge 2$, \  $1 \le p \le m-1$, \  $1 \le q \le  m+1 $  and  $p \equiv q [2]$.  \end{description}
  \end{theorem}
\begin{theorem} \label{sati} The characters of the discrete series are: 
\begin{displaymath} ch(L( c_{m} , h_{pq}^{m}))(t) = tr(t^{L_{0}-c_{m}/24})=\chi_{NS}(t).\Gamma^{m}_{pq}(t).t^{-c_{m}/24} \ \  \textrm{with}    \end{displaymath}
\begin{displaymath}\chi_{NS}(t) = \prod_{n \in \NNN^{\star}}\frac{1+t^{n-1/2}}{1-t^{n}},   \ \ \  \Gamma^{m}_{pq}(t) = \sum_{n \in \ZZZ}(t^{\gamma^{m}_{pq}(n)}-t^{\gamma^{m}_{-pq}(n)} )\ \  \textrm{and} \end{displaymath}
\begin{displaymath} \gamma^{m}_{pq}(n) =  \frac{[2m(m+2)n-(m+2)p+mq]^{2}-4}{8m(m+2)} \end{displaymath}   \end{theorem}

\subsection{Goddard-Kent-Olive framework} 
We take $\gg = \sl_{2}$. Let   $H$ an irreducible unitary, projective,  positive energy representation of the loop algebra $ L\gg$. We define the character of $H$ as:   
$ ch(H)(t,z) = tr( t^{L_{0}- \frac{C}{24}}z^{X_{3} }) $.  $L\gg$ acts on $\F_{NS}^{\gg}$, and by  Jacobi's triple product identity  $  \sum_{k \in \ZZZ} t^{\1/2 k^{2} }z^{k} =  \prod_{n \in \NNN^{\star}}(1+t^{n-\1/2}z)(1+t^{n-\1/2}z^{-1}) (1-t^{n})$, we prove that  $ch(\F_{NS}^{\gg})(t,z) = t^{-1/16}  \chi_{NS} (t)   \theta (t,z) $
with   $ \chi_{NS} (t) = \prod_{k \in \NNN^{\star} }(\frac{ 1+t^{n- \1/2 }}{1-t^{n} } ) $ and $   \theta (t,z) = \sum_{k \in \ZZZ} t^{\1/2 k^{2} }z^{k} $.
Hence, let $ H = L(j,\ell) $, and the theta functions $ \theta_{n,m}(t,z) = \sum_{k \in \frac{n}{2m} + \ZZZ }t^{mk^{2}}z^{mk},  $  then applying  the Weyl-Kac  formula  to $L\gg$:   $  ch(L(j,\ell)) = \frac{ \theta_{2j+1,\ell + 2}- \theta_{-2j-1,\ell + 2} }{ \theta_{1,2}- \theta_{-1,2}} $ (see \cite{5ab}, \cite{5b} or \cite{1} p 62). 
Now, adapting the proof in \cite{8} p 122, we obtain the product formula:  $ \theta(t,z).\theta_{p,m}(t,z) = \sum_{  \stackrel{0  \le q < 2(m+2)}{p \equiv q \lbrack 2 \rbrack}      }(\sum_{ n \in \ZZZ}t^{\alpha_{pq}^{m}(n)  } )\theta_{q,m+2}(t,z) $ with $   
 \alpha^{m}_{p,q}(n) =  \frac{[2m(m+2)n-(m+2)p+mq]^{2}}{8m(m+2)}.$ \\
Now,  $L\gg$ acts on $L(j,\ell) \otimes \F_{NS}^{\gg}$ at level $ \ell + 2$; we deduce:
$  ch( L(j,\ell) \otimes \F_{NS}^{\gg}) =  \sum_{  \stackrel{1  \le q \le m+1}{p \equiv q \lbrack 2 \rbrack}} F_{pq}^{m}. ch(L(k,\ell + 2)), \ 
 F_{pq}^{m} (t)  = t^{-1/16}   \chi_{NS} (t) \sum_{ n \in \ZZZ}(t^{\alpha_{p,q}^{m}(n)  }-t^{\alpha_{-p,q}^{m}(n)}),
 Ê\\  p=2j+1 ,  \  q=2k+1 $ and $   m = \ell + 2$; and the tensor product decomposition: \\Ê $  L(j,\ell) \otimes \F_{NS}^{\gg} = \bigoplus_{\stackrel{1  \le q \le m+1}{p \equiv q \lbrack 2 \rbrack} } M_{pq}^{m} \otimes  L(k,\ell + 2)  $   with  $M_{pq}^{m}$ the multiplicity space. \\ General GKO framework:  Let $\hh$ be Lie $\star$-superalgebra acting unitarily on a finite direct sum $H = \bigoplus_{}  M_{i} \otimes  H_{i}  $ with $  H_{i} $ irreducible and $ M_{i} $ the multiplicity space. Ê
We see that $M_{i}$ is the inner product space of superintertwiners $Hom_{\hh}(H_{i}, H)$.  Now, if $\dd$ is a Lie $\star$-superalgebra acting on $H$ and $H_{i}$  as  unitary, projective, positive energy representations, whose  difference ($\pi(D) - \sum \pi_{i}(D)$) supercommutes with $\hh$,  then, so is on $M_{i}$, with cocycle, the difference of the others.
Then, taking $\hh = \hat{\gg}$ and $\dd = \WW_{1/2}$, we find
 $  c_{M_{pq}^{m} } = \frac{dim(\gg)}{2} (1 - \frac{ 2g^{2}}{(\ell + g)(\ell + 2g) }  ) = \frac{3}{2} (1 - \frac{ 8}{m(m+2) }  )=: c_{m},  $
because  $m = \ell + 2 $, $ g = 2$ and $dim(\gg) = 3$. 
 Now, the character of  a $ \Vir_{1/2}$-module $H$ is :
$  ch(H)(t) = tr( t^{L_{0}- \frac{C}{24}}), $ then: 
$ ch(M_{pq}^{m})(t) =  t^{-\frac{c(m)}{24}}.\chi_{NS} (t).\Gamma_{pq}^{m} (t) $ with 
$ \Gamma^{m}_{pq}(t) = \sum_{n \in \ZZZ}(t^{\gamma^{m}_{pq}(n)}-t^{\gamma^{m}_{-pq}(n)} ),    \chi_{NS}(t) = \prod_{n \in \NNN^{\star}}\frac{1+t^{n-1/2}}{1-t^{n}} $ and 
$ \gamma^{m}_{pq}(n) =  \frac{[2m(m+2)n-(m+2)p+mq]^{2}-4}{8m(m+2)}. $ 
Hence, $h=h_{pq}^{m} = \frac{[(m+2)p-mq]^{2}-4}{8m(m+2)}$ is  the lowest eigenvalue of $L_{0}$ on $M_{pq}^{m}$;
let $(p',q') = (m-p, m+2-q)$, then: \\
$\begin{array}{c} \end{array}$  \hspace{3cm}
 $ch( M_{pq}^{m}   )  \sim t^{ - \frac{c_{m}}{24} }. \chi_{NS}(t).t^{h_{pq}^{m}  }.(1- t^{\frac{pq}{2}} - t^{\frac{p'q'}{2}}). 
$ \\ Hence, $ ch( M_{pq}^{m}   ) .  t^{ \frac{c_{m}}{24} } \sim   t^{h_{pq}^{m}  }$, and  the $h_{pq}^{m}$-eigenspace of $L_{0}$  is one-dimensional, so $L(c_{m},h_{pq}^{m})$ is a $ \Vir_{1/2}$-submodule of $M_{pq}^{m} $, and $ch( L(c_{m},h_{pq}^{m}) ) \le ch( M_{pq}^{m}   )$. \\
Finally, because $M_{pq}^{m} $ is unitary, so is for  $L(c_{m},h_{pq}^{m})$ on the discrete series.

\subsection{Kac determinant formula}  
From $(c_{m}, h_{pq}^{m})$,  we define $h_{pq}^{c}$, $\forall c \in \CCC$. Let $\varphi_{pp}(c,h) = (h-h_{pp}^{c})$, \\  
$\varphi_{pq}(c,h) = (h-h_{pq}^{c})(h-h_{qp}^{c})$ if $p \neq q$,  then  $\varphi_{pq} \in \CCC[c,h ]$ is irreducible. \\Ê Let $V_{n}(c,h)$ the $n$-eigenspace of $D=L_{0}-hI$ and $d(n)$ its dimension. \\ Let $M_{n}(c,h)$ the matrix of $( . , . ) $ on $V_{n}(c,h)$ and $det_{n}(c,h) = det( M_{n}(c,h))$. \\ For example,
$M_{0}(c,h) =  (\Omega , \Omega)  = (1) $, $M_{\1/2}(c,h) =  (G_{-\1/2}\Omega ,G_{-\1/2} \Omega)  = ( 2h) $, \\ $M_{1}(c,h) =   (L_{-1}\Omega ,L_{-1} \Omega)  = ( 2h) $, and  $M_{\3/2}(c,h) = $ \begin{center}$
 \left (  \begin{array}{cc}  (G_{-\1/2}L_{-1}\Omega ,G_{-\1/2}L_{-1} \Omega)    &   (G_{-\1/2}L_{-1}\Omega ,G_{-\3/2} \Omega)  \\
                                                        (G_{-\3/2}\Omega ,G_{-\1/2}L_{-1} \Omega)           &                 (G_{-\3/2}\Omega ,G_{-\3/2} \Omega)               \end{array}   \right )
     =    \left (  \begin{array}{cc}  2h + 4h^{2}   &   4h \\
                                                      4h        &          2h+\frac{2}{3}c             \end{array}   \right )     $ \end{center} Now, $det_{\3/2}(c_{m},h) =  8h[h^{2} - (\frac{3}{2} -\frac{c_{m}}{3})h +  c/6 ] = 8h (h - h_{13}^{m} )(h - h_{31}^{m} )  $, then, $det_{\3/2}(c,h) = 8h (h - h_{13}^{c} )(h - h_{31}^{c} )  = 8 \varphi_{11}(c,h) . \varphi_{13}(c,h) $ \ $\forall c \in \CCC$. \\  Hence,  others  examples permits to interpolate the Kac determinant formula: \begin{displaymath} det_{n}(c,h) = A_{n}  \prod_{\stackrel{0  < pq/2 \le n}{p \equiv q \lbrack 2 \rbrack} } (h - h_{pq}^{c} )^{d(n-pq/2)} = A_{n}  \prod_{\stackrel{0  < pq/2 \le n}{p \le q , \  p \equiv q \lbrack 2 \rbrack} } \varphi_{pq}^{d(n-pq/2)}(c,h)   \end{displaymath}
                                                      with $A_{n} > 0$ independent of $c$ and $h$. \\ 
To prove it, we will use singular vectors $s \in V(c,h)$, i.e. $ L_{0} . s = (h+n)s$  with $n>0$ its level, 
 and $\Vir_{1/2}^{+} . s = {0}$. This is equivalent to $G_{1/2}.s = G_{3/2}.s = 0$, and so we easily find $(mG_{-3/2} - (m+2)L_{-1}G_{-1/2})\Omega \in V_{3/2}(c_{m}, h_{13}^{m})$,  \\ $G_{-1/2}\Omega  \in V_{1/2}(c,h_{11}^{c})$, or  $ (L_{-1}^{2}-\frac{4}{3}h_{22}^{c}L_{-2}- G_{-3/2}G_{-1/2})\Omega \in V_{2}(c,h_{22}^{c})$. \\ Now, $ch(V(c,h )) = t^{h-\frac{c}{24}} \chi_{NS} (t)$ and the singular vectors generate $K(c,h)$. So, $V(c,h )$ admits a  singular vector of minimal level $n \in \1/2 \NNN$ if and only if  \begin{center}
$ch(L(c,h)) \sim  t^{h-\frac{c}{24}} \chi_{NS} (t) (1-t^{n})$. \end{center}  Now, thanks to GKO coset construction:  \begin{center} $ ch( L(c_{m},h_{pq}^{m}) ) \le ch( M_{pq}^{m}   )  \sim t^{ - \frac{c_{m}}{24} }. \chi_{NS}(t).t^{h_{pq}^{m}  }.(1- t^{\frac{pq}{2}} - t^{\frac{p'q'}{2}})$ \end{center}
So $V(c_{m},h_{pq}^{m})$ admits a singular vector $s$ at level $n' \le min(pq/2 , p'q'/2)$ and for $n > n'$, $det_{n}$  vanishes at $(c_{m},h_{pq}^{m})$ for $m$ sufficiently large integer. Then it vanishes at infinite many zeros of the irreducible $\varphi_{pq}$, which so $\varphi_{pq}$ divides $det_{n}$. But $s$ generates  a subspace of dimension $d(n-n')$ at level $n$,  so $d_{n}(c,h) = \prod_{\stackrel{0  < pq/2 \le n}{p \equiv q \lbrack 2 \rbrack} } (h - h_{pq}^{c} )^{d(n-pq/2)}$ divides  $det_{n}$. Finally,  a cardinality argument shows  $d_{n}$ and $det_{n}$,  with the same degree in $h$. The result follows.

\subsection{Friedan-Qiu-Shenker unitarity criterion }

The FQS criterion was discovered for $\Vir$ by Friedan, Qiu and Shenker \cite{4a}, but mathematicians estimated their proof too light, and then, in the same time, FQS \cite{4d} and Langlands \cite{9} published a complete proof. At the beginning of our research on  $\Vir_{1/2}$, we decided to adapt the way of Langlands, but we find a mistake in this paper (\cite{9} lemma 7b p 148: $p=2$,  $q=1$,  $m=2$,  $h_{pq}^{m}= \frac{5}{8}$,  $M=4$ or $p=4$, $q=1$, $m=3$, $h_{pq}^{m}= \frac{7}{2}$,  $M=13$ yield case $(B)$, but $(p,q) \ne (1,1)$ and $m \ngtr q+p-1$. In fact, we need to distinguish between $q \ne 1$ and $q=1$, but not between $(p,q)  \ne (1,1)$ and $q=  (1,1)$). Next, we discovered that Sauvageot  has ever published such an adaptation, without correction (\cite{10} lemma 2 (ii) p 648). Then, we chose  the way of FQS:

We are looking for a necessary condition on $ (c,h) $ for $V(c,h) $ has no ghost. First of all, if $V(c,h) $ admits no ghost then  $c, h \ge 0$ (easy). Now,  Kac determinant doesn't vanish on the region $h > 0$, $c >3/2 $, and for $(c,h)$ large, we prove that the form $( . , . ) $ is positive. So by continuity, if $h \ge 0$ and  $c  \ge 3/2 $, $V(c,h) $ admits no ghost. Now, on the region  $0 \le c < 3/2 $, $h\ge 0$ , the FQS criterion says that  $V(c,h) $ admits ghosts if  $ (c,h) $ does not belong to $(c_{m},h_{pq}^{m})$, with integers  $m \ge 2$, \  $1 \le p \le m-1$, \   $1 \le q \le  m+1 $  and  $p \equiv q [2]$, ie, exactly the discrete series given by GKO construction !  
To prove this result, we exploit the zero set of Kac determinants, constitutes by curves $C_{pq}$ of equation $h=h_{pq}^{c}$ with  $0 \ne p \equiv q [ 2 ]$. 
First of all, we restrict to $C'_{pq}$, the open subset of $C_{pq}$, between $c=3/2$ and its first intersection at level $pq/2$. Let $p'q' > pq$, $C_{p'q'}$ is a first intersector of $C'_{pq}$ if at level $p'q'/2$, it is the first to intersect $C'_{pq}$ starting from  $c=3/2$.  We see that all these first intersections constitutes exactly the discrete series.
Now, for each open region between the curves $C'_{pq}$, we can find $n$ with  $det_{n}$ negative on. This significate that $V(c,h) $ admits ghost on, and so we can eliminate these regions.   Hence now, we have to eliminate the intervals on $C'_{pq}$  between the points of the discrete series. We start from the no-ghost region $h > 0$, $c >3/2 $ and we go towards such an interval. On the way, we encounter a (well choosen) curve vanishing to order $1$; so on the other side, there is a ghost. We continue along the area of this curve with our ghost, up to an intersection point. Now, because the intersections are transversals, we can distinguish null vectors from the first curve to the second, and so our ghost continues to be a ghost on the other curve. Repeating this principle, we can go to the interval, without losing the ghost. Then,  FQS criterion and  theorem \ref{satya} follow.

\subsection{Wassermann's argument}  
We show that the multiplicity space of the coset construction, is an irreducible representation of the Neveu-Schwarz algebra, which (as in \cite{1} p 72 for $\Vir$) gives directly the characters on the discrete series without the Feigin-Fuchs resolution \cite{14}: 

As a corollary of FQS criterion's proof,  at levels $\le M = max( pq/2 , p'q'/2 )$, there exists only two singular vectors $s$ and $s'$, at levels $pq/2$ and $p'q'/2$. \\ 
Hence, $ch(L(c_{m} , h_{pq}^{m})) \sim   t^{h_{pq}^{m} - c_{m}/24}  \chi_{NS}(t)  (1 - t^{pq/2} - t^{p'q'/2})$, as for the multiplicity space $M_{pq}^{m}$,  and so
 $ch(M_{pq}^{m})-ch(L(c_{m} , h_{pq}^{m})) =   \chi_{NS}(t) . t^{- c_{m}/24} o( t^{h_{pq}^{m} + M})$. Now, we know that $L(c_{m} , h_{pq}^{m})$ is a submodule of $M_{pq}^{m}$; if  $M_{pq}^{m}$ admits an other irreducible submodule, by FQS criterion, it is of the form $L(c_{m} , h_{rs}^{m})$; but through the lemma: $h_{pq}^{m} + M > m^{2}/8$ and  $h_{rs}^{m} \le  \frac{ m(m-2)}{8 }$, we obtain, by coherence on the characters,  the contradiction: 
$\frac{ m^{2}}{8 } < M+h_{pq}^{m} < h_{rs}^{m} \le  \frac{ m(m-2)}{8 }$.  Then, Ê$M_{pq}^{m} = L(c(m),h_{p,q}^{m})$ and  $ch(L(c_{m} , h_{pq}^{m})) =  ch(M_{pq}^{m})$, but the characters of the multiplicity spaces are ever known by GKO.  The theorem \ref{sati}  follows. 

\newpage

\section{Goddard-Kent-Olive framework}
\subsection{Characters of $ L\gg$-modules}
In  this section, we take  $\gg = \sl_{2}$.   
 Let  $H$ a unitary, projective and positive energy representation of the loop algebra $ L\gg$ (see section \ref{loop}). 
 \begin{remark} Thanks to $\gg  \hookrightarrow L\gg :  \   X_{a} \mapsto  X^{a}_{0}$,  $\gg$  acts on $H$, \\Ê and by the previous work, the Virasoro algebra $\Vir$ acts on too: \\ 
 $  [L_{m} , L_{n} ] = (m-n) L_{m+n}  +  \frac{C}{12} m(m^{2}-1) \delta_{m+n}  \quad   \quad ( n \in \ZZZ, \ C \  \textrm{central} )$.  \end{remark}
\begin{definition}
A character of $H$ as $ L\gg$-module is definied by:
\begin{displaymath}  ch(H)(t,z) = tr( t^{L_{0}- \frac{C}{24}}z^{X_{3} }) 
\end{displaymath}
\end{definition}
\begin{lemma} (Jacobi's triple product identity) 
\begin{displaymath} 
 \sum_{k \in \ZZZ} t^{\1/2 k^{2} }z^{k} =  \prod_{n \in \NNN^{\star}}(1+t^{n-\1/2}z)(1+t^{n-\1/2}z^{-1}) (1-t^{n})
 \end{displaymath} 
\end{lemma}
\begin{proof}  See \cite{1} p 62.  \end{proof}
\begin{remark} On section 4.2.1 of \cite{NSOAI}, $L \gg$ acts on $\F_{NS}^{\gg}$, with $\pi_{\F_{NS}^{\gg}}(X_{3}) = S_{0}^{3}$.  \end{remark}
\begin{proposition}  \label{ch2} \quad $ ch(\F_{NS}^{\gg})(t,z) = t^{-1/16}  \chi_{NS} (t)   \theta (t,z) \quad \textrm{with} $
\begin{displaymath}
\chi_{NS} (t) = \prod_{n \in \NNN^{\star} }(\frac{ 1+t^{n- \1/2 }}{1-t^{n} } ) \quad \textrm{and} \quad  \theta (t,z) = \sum_{k \in \ZZZ} t^{\1/2 k^{2} }z^{k} \end{displaymath}
\end{proposition}
\begin{proof} 
  $C$ acts as multiplicative constant $c_{\F_{NS}^{\gg}}=  \frac{dim(\gg)}{2} =   \frac{3}{2} $, so,  $- \frac{c}{24} = -1/16  $ \\
$ [S_{m}^{a} , \psi_{n }^{b} ] = i\sum_{c} \Gamma_{ab}^{c}  \psi^{c}_{m+n}$, so, $[S_{0}^{3} , \psi_{n }^{3} ] =0$,  $[S_{0}^{3} , \psi_{n }^{1} ] = i \psi^{2}_{n}$, $[S_{0}^{3} , \psi_{n }^{2} ] = -i \psi^{1}_{n}$.  
 Let  $\varphi^{3}_{n} = \psi_{n }^{3}$,   $\varphi^{1}_{n} = i\psi^{1}_{n} - \psi_{n }^{2} $,   $\varphi^{2}_{n} = \psi^{1}_{n} - i\psi_{n }^{2} $, then, 
 $[S_{0}^{3} , \varphi_{n }^{3} ] =0$,  $[S_{0}^{3} , \varphi_{n }^{1} ] =  \varphi^{1}_{n}$ and $[S_{0}^{3} , \varphi_{n }^{2} ] = - \varphi^{2}_{n}$. 
 Now, if $M = PDP^{-1} $, then, $tr(M) = tr(D) $ and $tr(z^{M}) = tr(z^{D}) $, but, $ad_{S_{0}^{3}}$ acts diagonally on $\widehat{\gg}_{-}$ with basis ($\varphi_{n}^{i}$), \\  $[L_{0} , \varphi_{m }^{i} ] = -m \varphi_{m }^{i} $, 
and $S_{0}^{3} \Omega = 0$, so, it suffices to associate: \\  $t^{n-\1/2}$ to $\varphi_{-n +\1/2 }^{3} $, \  $t^{n-\1/2}z$ to $\varphi_{-n + \1/2 }^{1} $, and   $t^{n-\1/2}z^{-1}$ to $\varphi_{-n + \1/2 }^{2} $ to find: 
\begin{displaymath}
ch(\F_{NS}^{\gg})(t,z) = t^{-1/16} \prod_{n \in \NNN^{\star}} (1+t^{n-\1/2})(1+t^{n-\1/2}z)(1+t^{n-\1/2}z^{-1})
\end{displaymath} 
The result follows by the Jacobi's triple product identity.  
 \end{proof}
 \begin{definition}   Let $m \in \NNN^{\star} $, $n \in \ZZZ $, $t, z \in \CCC$ with $\Vert t  \Vert   < 1$. \\  Let the theta functions:
 \begin{displaymath} \theta_{n,m}(t,z) = \sum_{k \in \frac{n}{2m} + \ZZZ }t^{mk^{2}}z^{mk}  \end{displaymath} 
 \end{definition}
 \begin{theorem} \label{ch1} Let $ H = L(j,\ell)$,   irreducible representation of $L\sl_{2}$, then
 \begin{displaymath}   ch(L(j,\ell)) = \frac{ \theta_{2j+1,\ell + 2}- \theta_{-2j-1,\ell + 2} }{ \theta_{1,2}- \theta_{-1,2}} \end{displaymath} 
 \end{theorem}
\begin{proof}  An application of the Weyl-Kac character formula  to $L\sl_{2}$ \\  (see \cite{5ab}, \cite{5b} or \cite{1} p 62).    \end{proof}
\begin{proposition} (Product formula)  \label{prod}
 \begin{displaymath} \theta(t,z).\theta_{p,m}(t,z) = \sum_{  \stackrel{0  \le q < 2(m+2)}{p \equiv q \lbrack 2 \rbrack}      }(\sum_{ n \in \ZZZ}t^{\alpha_{pq}^{m}(n)  } )\theta_{q,m+2}(t,z)  \end{displaymath} 
  \begin{displaymath}  
  \textrm{with} \quad \quad  \alpha^{m}_{p,q}(n) =  \frac{[2m(m+2)n-(m+2)p+mq]^{2}}{8m(m+2)}
     \end{displaymath} 
\end{proposition}
\begin{proof}  We adapt the proof in \cite{8c} or \cite{8} p 122, to the super case:
\begin{displaymath}  \theta(t,z).\theta_{p,m}(t,z) = \sum_{k , k'} t^{\1/2 k^{2}+mk'^{2}}z^{k+mk'} \end{displaymath}
Let $k=i$, $k' = \frac{p}{2m} + i' $ where $i$, $i' \in \ZZZ$; we define $s$, $s'$ by: 
\begin{itemize}
\item   $(m+2)s = k-2k' = i-2i'- \frac{p}{m}  $
\item   $(m+2)s' = k+mk' = (m+2)(k'+s) $
\end{itemize}
Now, $p+2(i-2i') = 2(m+2)n + q$ with $ 0  \le q < 2(m+2)$,  \ $p \equiv  q  \lbrack  2  \rbrack$ , then: 
\begin{displaymath} 
s = n -\frac{(m+2)p - mq }{2m(m+2)}    \hspace{0,3cm}    \textrm{and}   \hspace{0,3cm}  s' = n'+ \frac{q }{2(m+2)}  \hspace{0,3cm}   n, n' \in \ZZZ  \hspace{0,3cm} ( \textrm{with}  \hspace{0,3cm}  n'=n+i' ).
\end{displaymath} 
This gives  a bijection between pairs $(k,k')$ and triples $(q,s,s')$. \\
Now, $\1/2 k^{2}+mk'^{2} = \1/2 (ms + 2s')^{2} + m(s-s')^{2} = \1/2  m(m+2)s^{2} + (m+2) s'^{2}  $ \\Ê and 
$\1/2  m(m+2)s^{2} = \1/2  m(m+2)(n -\frac{(m+2)p - mq }{2m(m+2)} )^{2} =   \alpha^{m}_{p,q}(n) $ \end{proof}
\begin{remark}  \label{rk}
On \cite{NSOAI} section 4.2.3, $L\gg$ acts on $ \F_{NS}^{\gg}  \otimes L(j,\ell)$ as unitary,  projective, positive energy representation of  level $ \ell + 2$  (see \cite{NSOAI} def. 4.36).   \end{remark}
\begin{corollary}  \label{cor}  Let $p=2j+1 $,   $q=2k+1 $ and $m = \ell + 2$, then:
\begin{displaymath}   ch(\F_{NS}^{\gg}  \otimes L(j,\ell)) =  \sum_{  \stackrel{1  \le q \le m+1}{p \equiv q \lbrack 2 \rbrack}      } F_{pq}^{m}. ch(L(k,\ell + 2))
\end{displaymath}
\begin{displaymath} 
   \textrm{with} \quad \quad   F_{pq}^{m} (t)  = t^{-1/16}  \chi_{NS} (t) \sum_{ n \in \ZZZ}(t^{\alpha_{p,q}^{m}(n)  }-t^{\alpha_{-p,q}^{m}(n)})
\end{displaymath}
\end{corollary} 

We apply theorem \ref{ch1},  propositions \ref{ch2} and

\begin{proof} 
$L\gg$ acts on $H$ as $(I \otimes X + X \otimes I )$, then: \\ 
 $ch(\F_{NS}^{\gg}  \otimes L(j,\ell)) =ch(  \F_{NS}^{\gg}).ch( L(j,\ell))$; now by proposition  \ref{prod}: 
 \begin{displaymath} \theta(t,z).(\theta_{p,m}(t,z) - \theta_{-p,m}(t,z) ) = \sum_{  \stackrel{0  \le q < 2(m+2)}{p \equiv q \lbrack 2 \rbrack}      }(\sum_{ n \in \ZZZ}t^{\alpha_{pq}^{m}(n)  } -t^{\alpha_{-p,q}^{m}(n)})\theta_{q,m+2}(t,z)  \end{displaymath}

   But for $m+2 \le q' < 2(m+2) $,   $q'= 2(m+2)-q $ with $1 \le q \le m+2 $. Now by symmetry, $ \theta_{2(m+2)-q,m + 2}  =  \theta_{-q,m + 2}$,  and  $F_{p, 2(m+2)-q}^{m} = -F_{pq}^{m}$ because  $\alpha_{p,2(m+2)-q}^{m}(n) = \alpha_{-p,q}^{m}(-n-1)$.
Finally,  $F_{p0}^{m} = F_{p, m+2}^{m} = 0$ because \\ $\alpha_{p,0}^{m}(n) = \alpha_{-p,0}^{m}(-n)$ and Ê $\alpha_{p,m+2}^{m}(n) = \alpha_{-p,m+2}^{m}(-n-1)$; the result follows.
\end{proof}  
\begin{corollary}  \label{mult}  (Tensor product decomposition)
\begin{displaymath}  \F_{NS}^{\gg}  \otimes L(j,\ell)= \bigoplus_{\stackrel{1  \le q \le m+1}{p \equiv q \lbrack 2 \rbrack} } M_{pq}^{m} \otimes  L(k,\ell + 2) 
\end{displaymath}  with  $M_{pq}^{m}$ the multiplicity space.
\end{corollary}
\begin{proof}
By complete reducibility and remark \ref{rk}, $ \F_{NS}^{\gg} \otimes L(j,\ell) $ is a direct sum of irreducibles of type $L(k, \ell + 2) $; the result follows by  corollary \ref{cor}. 
\end{proof} 
  \begin{corollary}  \label{mult56}  As $\hat{\gg} = \hat{\gg}_{+} \ltimes \hat{\gg}_{-}$ representations, we obtain;   
  \begin{displaymath} \F_{NS}^{\gg} \otimes ( L(j,\ell) \otimes \F_{NS}^{\gg}) = \bigoplus_{\stackrel{1  \le q \le m+1}{p \equiv q \lbrack 2 \rbrack} } M_{pq}^{m} \otimes  ( L(k,\ell + 2) \otimes \F_{NS}^{\gg}) 
\end{displaymath}   \end{corollary}
\begin{proof}  Recall proposition 4.35 and remark 4.36 of \cite{NSOAI}.  \\ Next, the characters of $\hat{\gg}$-modules are defined as for $ \hat{\gg}_{+} $-modules.    \end{proof}  
\subsection{Coset construction} \label{52}
\subsubsection{General framework}
Let $\hh$ be a Lie $\star$-superalgebra acting unitarily on an inner product space $H$, a direct sum of irreducibles of finitely many isomorphic type $H_{i}$: 
\begin{displaymath}  H = \bigoplus_{i}  M_{i} \otimes  H_{i}   \quad \textrm{ with  $M_{i}$ the multiplicity space.}  \end{displaymath}
\begin{remark}  $\hh$ acts on $H$ as  $\pi(X) = \sum I \otimes \pi_{i}(X)$.   \end{remark}
\begin{definition}
Let  $K_{i} = Hom_{\hh}(H_{i}, H) $, \\  space of homomorphisms that supercommute with $\hh$ (graded intertwinners).
\end{definition}
\begin{reminder} 
 $Hom_{\hh}(H_{i}, H_{j}) = \delta_{ij}\CCC$, $End_{\hh}(H) = \bigoplus End(M_{i})\otimes \CCC $.
\end{reminder} 
\begin{lemma}
$K_{i}$ admits a natural inner product. 
\end{lemma}
\begin{proof} 
If $S, T \in K_{i}$, then $T^{\star}S \in End_{\hh}(H_{i}) = \CCC $, and so, $(S,T)= T^{\star}S$ \\  defines the inner product.
\end{proof}
\begin{lemma}
$\rho :   \bigoplus  K_{i} \otimes  H_{i}  \to  H $ such that: \quad  $\rho (\sum \xi_{i}\otimes \eta_{i} ) = \sum \xi_{i}( \eta_{i})$, \\ is a unitary isomorphism of $\hh$-modules.
\end{lemma}
\begin{proof} 
Let $\sum m_{i}\otimes \eta_{i} \in H$ and $\xi_{i} :  \eta_{i} \mapsto  m_{i}\otimes \eta_{i} $, then $\xi_{i} \in K_{i}$, \\  because $\hh$ acts on $H$ as $\sum I \otimes \pi_{i}$; and so, $\rho$ is surjective. \\
Now, $(\rho (\sum \xi'_{i}\otimes \eta'_{i} )  , \rho (\sum \xi_{j}\otimes \eta_{j} )  ) = \sum( \xi'_{i}( \eta'_{i} )  ,  \xi_{j}(\eta_{j}))  =  \\ \sum (  \xi_{j}^{\star}\xi'_{i}( \eta'_{i} )  ,\eta_{j})=
\sum (  \xi_{j}^{\star} , \xi'_{i})  ( \eta'_{i}  ,\eta_{j}) = (\sum \xi'_{i}\otimes \eta'_{i}  , \sum \xi_{j}\otimes \eta_{j}  ) $   \end{proof}
\begin{remark}
An operator $A$ on $H$ which supercommutes with $\hh$, acts by definition,  on each $K_{i}$ by an $A_{i}$, and, identifying $M_{i}$ and $K_{i}$, \   $A = \sum A_{i} \otimes I $
\end{remark}
 Let $\dd$ be a Lie $\star$-superalgebra acting as $\pi (D)$ on $H$, and as $\pi_{i}(D)$ on $H_{i}$.
\begin{corollary}  \label{ballon}
If $\forall D \in \dd $,  $\sigma (D)  =  \pi (D) - \sum I \otimes \pi_{i}(D)$ supercommutes with $\hh$, then $\dd$ acts on $M_{i}$ as $\sigma_{i}(D) $ with  $\sigma (D)  = \sum \sigma_{i}(D) \otimes I $.
\end{corollary} 
\begin{definition}  Let   $B_{F} (D_{1}, D_{2}) := [\pi_{F}(D_{1}), \pi_{F}(D_{2}) ]  - \pi_{F}  [D_{1}, D_{2} ] $.    \end{definition} 
 
 \begin{remark} If $F$ is  unitary, projective and positive energy  (see definition \ref{proj}), 
 the cocycle $b_{F}$ is defined by  $B_{F} (D_{1}, D_{2})= b_{F} (D_{1}, D_{2})I_{F}$.   \end{remark}
\begin{proposition} \label{cocycle}If in addition to corollary \ref{ballon}, $\pi $ and $\pi_{i}$ are unitary, projective, positive energy representations, then, so is $\sigma_{i}$, and the cocycle of $\dd$ on $M_{i}$ is the difference of the cocycles on $H$ and on $H_{i}$.
\end{proposition}
\begin{proof}
$\pi = \sum (I \otimes \pi_{i}+ \sigma_{i}\otimes I )$ and  $B_{H} = \sum(  I \otimes B_{H_{i}}   + B_{M_{i}}  \otimes I  )  $.  \\
    $M_{i}\otimes H_{i}  \subset H$, so, $b_{H}I  = b_{M_{i} \otimes H_{i}}I =  I \otimes B_{H_{i}}   + B_{M_{i}}  \otimes I$. \\
 Finally,   $  B_{M_{i}} \otimes I = b_{H}I-   I \otimes B_{H_{i}}  =  ( b_{H} -  b_{H_{i}} ) I \otimes I$  \end{proof}
\subsubsection{Application}
We apply the previous result to corollary \ref{mult56} with  $\hh = \hat{\gg}$ and $\dd = \WW_{1/2}$.

 \begin{convention} To have a graded Lie bracket coherent with tensor product, we need to introduce the following convention:  let $A$, $B$ be superalgebras, then, the product on $A \otimes B$ is defined as follows: 
  \begin{center}$ (a\otimes b ).(c \otimes d) = (-1)^{\varepsilon(b)\varepsilon(c)}ac \otimes bd$  \ with  $\varepsilon(b)$ ,   $\varepsilon(c) \in \ZZZ_{2} $ \end{center}  \end{convention}
  \begin{lemma} \label{sst} Let $\tt$ be a Lie superalgebra, then, by the previous convention:  
  \begin{center} $[ X \otimes I + I \otimes X  , Y \otimes I + I \otimes Y  ]_{\varepsilon} = [X,Y]_{\varepsilon} \otimes I + I\otimes [X,Y]_{\varepsilon}  $ \end{center}  
   \end{lemma}
\begin{corollary}
The Witt superalgebra $\WW_{1/2}$ acts on the multiplicity space $M_{pq}^{m} $ as  unitary, projective and positive energy representation, with central charge, 
 \begin{displaymath}  c_{M_{pq}^{m} } = \frac{dim(\gg)}{2} (1 - \frac{ 2g^{2}}{(\ell + g)(\ell + 2g) }  ) = \frac{3}{2} (1 - \frac{ 8}{m(m+2) }  )  
 \end{displaymath}
 $m = \ell + 2 $, $ g = 2$ and $dim(\gg) = 3$.
\end{corollary}
\begin{proof}
$\WW_{1/2}$ acts as $\sum I \otimes X$ on $\bigoplus M_{pq}^{m} \otimes  (L(k,\ell + 2) \otimes \F_{NS}^{\gg})$, as $X \otimes I + I \otimes X $ on $\F_{NS}^{\gg}\otimes( L(j,\ell) \otimes \F_{NS}^{\gg}))$, it's projective thanks to  lemma \ref{sst}, unitary,  positive energy,  and their difference supercommutes with $\hat{\gg}$ by proposition \ref{susybf}.  
Now by proposition \ref{cocycle}: 
\begin{center}
 $ c_{M_{pq}^{m}} =c_{\F_{NS}^{\gg}\otimes (L(j,\ell) \otimes \F_{NS}^{\gg}) } - ( c_{L(k,\ell + 2) \otimes \F_{NS}^{\gg}) }  ) 
  = c_{\F_{NS}^{\gg} } +  c_{L(j,\ell)  } +c_{\F_{NS}^{\gg} } -  (c_{L(k,\ell + 2) }+c_{\F_{NS}^{\gg} })
  =\frac{3}{2}  \cdot \frac{\ell +\t1 g}{\ell + g} dim(\gg) -  \frac{\ell + g}{\ell + 2g} dim(\gg)$ \end{center} \end{proof} 
  \begin{remark} Let $\hat{\gg} \subset \hat{\gg} \oplus \hat{\gg}$ be the diagonal inclusion, then the previous construction is equivalent to the Kac-Todorov one \cite{8b}: the coset action of $\Vir_{1/2}$ is given by   $L^{\hat{\gg} \oplus \hat{\gg}}_{n} - L^{ \hat{\gg}}_{n} $ and $G^{\hat{\gg} \oplus \hat{\gg}}_{r} - G^{ \hat{\gg}}_{r} $. There exists also an manner to write this action only with ordinary loop algebra, due to Goddard,  Kent, Olive \cite{3a} (used and discussed in \cite{NSOAIII} section 2.7).   \end{remark}

\subsection{Character of the multiplicity space}  \label{chari}

\begin{definition}  \label{formidable}
$ \Vir_{1/2}$-module's character is  $ch(H)(t) = tr( t^{L_{0}- \frac{C}{24}}) $.

 \end{definition}
\begin{corollary} (Character of the multiplicity space) \label{space}
\begin{displaymath} ch(M_{pq}^{m})(t) = t^{-\frac{c(m)}{24}}.\chi_{NS} (t).\Gamma_{pq}^{m} (t)  \ \  \textrm{with}    \end{displaymath}
\begin{displaymath} \Gamma^{m}_{pq}(t) = \sum_{n \in \ZZZ}(t^{\gamma^{m}_{pq}(n)}-t^{\gamma^{m}_{-pq}(n)} ),  \ \ \ \chi_{NS}(t) = \prod_{n \in \NNN^{\star}}\frac{1+t^{n-1/2}}{1-t^{n}}  \ \  \textrm{and} \end{displaymath}
\begin{displaymath} \gamma^{m}_{pq}(n) =  \frac{[2m(m+2)n-(m+2)p+mq]^{2}-4}{8m(m+2)} \end{displaymath} 
\end{corollary}  
\begin{proof}
It follows by corollaries \ref{cor}, \ref{mult}, and,   $\gamma_{pq}^{m}(n) =  \alpha_{pq}^{m}(n) - \frac{1}{16} + \frac{c_{m}}{24}$. 
\end{proof}
\begin{lemma}
The lowest eigenvalue of $L_{0}$ on $M_{pq}^{m}$  is: 
 \begin{displaymath}  h=h_{pq}^{m} = \frac{[(m+2)p-mq]^{2}-4}{8m(m+2)}
\end{displaymath} 
\end{lemma}
\begin{proof}
$\chi_{NS} (t) \sim 1+t^{\1/2} $ and $min \{ \gamma^{m}_{pq}(n) , \gamma^{m}_{-pq}(n), n \in \ZZZ \} = \gamma^{m}_{pq}(0)=h_{p,q}^{m}$ 
\end{proof}
\begin{lemma}  \label{lem}
Let $(p',q') = (m-p, m+2-q)$, then:
\begin{displaymath}   ch( M_{pq}^{m}   )  \sim t^{ - \frac{c_{m}}{24} }. \chi_{NS}(t).t^{h_{pq}^{m}  }.(1- t^{\frac{pq}{2}} - t^{\frac{p'q'}{2}})
\end{displaymath}
\end{lemma}
\begin{proof}$\gamma^{m}_{-pq}(0) = \gamma^{m}_{pq}(0) + \frac{pq}{2}$, $\gamma^{m}_{-pq}(-1) = \gamma^{m}_{pq}(0) + \frac{p'q'}{2}$; and,
$\gamma^{m}_{pq}(0)$, $\gamma^{m}_{-pq}(0)$, $\gamma^{m}_{-pq}(-1)$ are the three lowest numbers of  $\{ \gamma^{m}_{pq}(n) , \gamma^{m}_{-pq}(n), n \in \ZZZ \}$.
\end{proof} 
\begin{corollary} \label{coro} $L(c_{m},h_{pq}^{m})$ is a $ \Vir_{1/2}$-submodule of $M_{pq}^{m} $
\end{corollary}
\begin{proof}  $ ch( M_{pq}^{m}   ) .  t^{ \frac{c_{m}}{24} } \sim  t^{h_{pq}^{m}  }$, then, the $h_{pq}^{m}$-eigenspace of $L_{0}$  is one-dimensional; $L(c(m),h_{pq}^{m})$ is the minimal $ \Vir_{1/2}$-submodule of $M_{pq}^{m} $ containing it.
\end{proof}
\begin{corollary} \label{cha}
$ ch( L(c_{m},h_{pq}^{m}) ) \le ch( M_{pq}^{m}   )  \sim t^{ h_{pq}^{m} - \frac{c_{m}}{24} }. \chi_{NS}(t)(1- t^{\frac{pq}{2}} - t^{\frac{p'q'}{2}})
$
\end{corollary}
\begin{theorem}  (Unitarity sufficient condition) \\  \label{suff}
 Let integers  $m \ge 2$, \  $1 \le p \le m-1$, \   $1 \le q \le  m+1 $  and  $p \equiv q [2]$, then: \\  
  $L(c_{m},h_{pq}^{m}) $ is a unitary highest weight representation of $ \Vir_{1/2}$
\end{theorem}
\begin{proof} Recall deÞnitions 2.5 and 2.21 of \cite{NSOAI}. \\
$M_{pq}^{m} $ is unitary; so is its $ \Vir_{1/2}$-submodule $L(c_{m},h_{pq}^{m}) $.
\end{proof}
\begin{remark}FQS  criterion  proves this is all its discrete series.
\end{remark} 
\newpage 
\section{Kac determinant formula}
\subsection{Preliminaries}  \label{pre}
Let $c$, $h \in \CCC$, recall  section 2.3  of  \cite{NSOAI} for definitions of Verma module $V(c,h)$, sesquilinear form $( . , . )$ and maximal proper submodule $K(c,h)$. \\
Let $(c,h) = (c_{m}, h_{pq}^{m}) =(\frac{3}{2} (1 - \frac{ 8}{m(m+2) } ) ,  \frac{[(m+2)p-mq]^{2}-4}{8m(m+2)}) $.
\begin{lemma}$ h_{pq}^{m} +   h_{qp}^{m} = \frac{p^{2}+q^{2}-2}{16}(1-2c_{m}/3) + \frac{(p-q)^{2}}{4}$  and $h_{pq}^{m}. h_{qp}^{m} = $ \\
 $ \frac{ 1}{16^{2} }[ 2(p-q)^{2}- (1-2c_{m}/3) (pq-p-q-1)] . [ 2(p-q)^{2}- (1-2c_{m}/3) (pq+p+q+1)]$
\end{lemma}   
Then, solving the system of the lemma, we can  define $h_{pq}^{c}$, $\forall c \in \CCC$.
\begin{definition}
$\varphi_{pp}(c,h) = (h-h_{pp}^{c})$ and \\ $\varphi_{pq}(c,h) = (h-h_{pq}^{c})(h-h_{qp}^{c})$ if $p \neq q$
\end{definition}
\begin{lemma}$ \varphi_{pq} \in \CCC[c,h ]$ is irreducible.
\end{lemma}  
\begin{definition}
Let $V_{n}(c,h)$ the $n$-eigenspace of $D=L_{0}-hId$ generated by the vectors
 $ G_{-j_{\beta}} \dots G_{-j_{1}}L_{-i_{\alpha}} \dots L_{-i_{1}}\Omega $  \  \  such that $\sum i_{s} + \sum j_{s} = n$, \\
with \  $   Ê0 <  i_{1} \le \dots \le i_{\alpha},   \quad  \frac{1}{2} \le j_{1} < \dots < j_{\beta}$; let $d(n)$ its dimension. \end{definition}
\begin{remark}  $d(n) < \infty$,  $d(n) =0$ for $n<0$. \\  Clearly  $(V_{n}(c,h),V_{n'}(c,h)) = 0$ if $n \neq n'$ and $V(c,h) = \bigoplus V_{n}(c,h) $. \end{remark}
\begin{definition}
Let $M_{n}(c,h)$ the matrix of $( . , . ) $ on $V_{n}(c,h)$ \\ and $det_{n}(c,h) = det( M_{n}(c,h))$
\end{definition}
\begin{examples}
$M_{0}(c,h) =  (\Omega , \Omega)  = (1) $, $M_{\1/2}(c,h) =  (G_{-\1/2}\Omega ,G_{-\1/2} \Omega)  = ( 2h) $, \\ $M_{1}(c,h) =   (L_{-1}\Omega ,L_{-1} \Omega)  = ( 2h) $, and, $M_{\3/2}(c,h) =$
\begin{displaymath} \left (  \begin{array}{cc}  (G_{-\1/2}L_{-1}\Omega ,G_{-\1/2}L_{-1} \Omega)    &   (G_{-\1/2}L_{-1}\Omega ,G_{-\3/2} \Omega)  \\
                                                        (G_{-\3/2}\Omega ,G_{-\1/2}L_{-1} \Omega)           &                 (G_{-\3/2}\Omega ,G_{-\3/2} \Omega)               \end{array}   \right )
     =    \left (  \begin{array}{cc}  2h + 4h^{2}   &   4h \\
                                                      4h        &          2h+\frac{2}{3}c             \end{array}   \right )                                                 
\end{displaymath}
\end{examples}
\begin{remark}
$det_{\3/2}(c_{m},h) =  8h[h^{2} - (\frac{3}{2} -\frac{c_{m}}{3})h +  c/6 ] = 8h (h - h_{13}^{m} )(h - h_{31}^{m} )  $, \\
 then, $det_{\3/2}(c,h) = 8h (h - h_{13}^{c} )(h - h_{31}^{c} )  = 8 \varphi_{11}(c,h) . \varphi_{13}(c,h) $ \ $\forall c \in \CCC$
\end{remark}
\begin{theorem}  (Kac determinant formula) \label{kac}
\begin{displaymath}
det_{n}(c,h) = A_{n}  \prod_{\stackrel{0  < pq/2 \le n}{p \equiv q \lbrack 2 \rbrack} } (h - h_{pq}^{c} )^{d(n-pq/2)} = A_{n}  \prod_{\stackrel{0  < pq/2 \le n}{p \le q , \  p \equiv q \lbrack 2 \rbrack} } \varphi_{pq}^{d(n-pq/2)}(c,h) 
\end{displaymath}
with $A_{n} > 0$ independent of $c$ and $h$.
\end{theorem}

\subsection{Singulars vectors and characters}   \label{singulars}  
\begin{definition}
A vector $s \in V(c,h)$ is singular if: 
\begin{description}
\item[(a)] $  L_{0} . s = (h+n)s   \quad    \textrm{with} \  n>0$  (its level)
\item[(b)] $\Vir_{1/2}^{+} . s = {0}$   \quad    \textrm{(recall definition 2.13 of \cite{NSOAI})}
\end{description} 
\end{definition}
\begin{remark} \label{shortdef} Let $ n > 0 $,   $s \in  V_{n}(c,h)$     is singular iff  $G_{1/2}.s = G_{3/2}.s = 0$  \end{remark}
  \begin{examples} \label{singulars} $(mG_{-3/2} - (m+2)L_{-1}G_{-1/2})\Omega \in V_{3/2}(c_{m},h_{13}^{m})$, \\ $G_{-1/2}\Omega  \in V_{1/2}(c,h_{11}^{c})$, \ $ (L_{-1}^{2}-\frac{4}{3}h_{22}^{c}L_{-2}- G_{-3/2}G_{-1/2})\Omega \in V_{2}(c,h_{22}^{c})$
    \end{examples}
\begin{definition}
 $K_{n}(c,h) = ker (M_{n}(c,h) ) =  \{ x \in V_{n}(c,h) ; (x,y) = 0 \  \forall y   \}$
\end{definition}    
\begin{proposition}  \label{cheval}
The singular vectors generate $K(c,h)$.
\end{proposition}
\begin{proof}
They clearly generate a subspace of $K(c,h)$. 
Now, let $v \in K_{n}(c,h)$, then $\Vir_{1/2}^{+}.v$ is of level $< n$ and $\exists n'$ such that  $(\Vir_{1/2}^{+})^{n'+1}.v = \{0 \}$ and  
 $(\Vir_{1/2}^{+})^{n'}.v \neq \{0 \}$ and contains a singular vector generating $v$.
\end{proof}
\begin{definition}
Let $V^{s}(c,h)$ the minimal $\Vir_{1/2}$-submodule of $V(c,h)$ \\ containing $s$   and $V_{n}^{s}(c,h) = V^{s}(c,h) \cap V_{n}(c,h) $. \end{definition}
\begin{lemma}
Let  $s$ singular of level $n'$, then $dim(V_{n}^{s}(c,h)) = d(n-n')$.
\end{lemma}
\begin{proof}  $D.(A.s) = n A.s$ $\iff$ $D.(A\Omega) = (n-n') A\Omega$
\end{proof}
\begin{lemma} \label{lem22}$ch(V(c,h )) = t^{h-\frac{c}{24}} \chi_{NS} (t)$
\end{lemma}
\begin{proof}
$ch(V(c,h )) =  tr( t^{L_{0}- \frac{c}{24}}) =  t^{h-\frac{c}{24}} \sum_{m \in \1/2\NNN} d(m)t^{m} $ \\
$ \chi_{NS} (t) = \prod_{n \in \NNN^{\star}} (\frac{1+q^{n-\1/2} }{1-q^{n}}) =  \prod_{n \in \NNN^{\star}} (1+q^{n-\1/2} )(1+q^{n} + q^{2n}+... )$ \\
 Identifying $q^{n-\1/2}$ to $G_{n-\1/2}$,  $q^{n} $ to $L_{n}$, the  coefficient of $q^{m} $ is exactly $d(m)$.
\end{proof}
\begin{corollary}
$ch(V^{s}(c,h )) = t^{n+h-\frac{c}{24}} \chi_{NS} (t)$, \  with $n$ the level of $s$.
\end{corollary}
\begin{remark}  \label{coro2}$dim(L_{n}(c,h) )= dim(V_{n}(c,h))-dim(K_{n}(c,h))$, then, \\ $ch(L(c,h)) = ch(V(c,h)) - \sum_{sÊ}ch(V^{s}(c,h)) +  \sum_{s, s'}ch(V^{s} \cap V^{s'}) - \dots$.  \end{remark}
\begin{corollary}  \label{coro3} $V(c,h )$ admits a  singular vector $s$ of minimal level $n$  if and only if 
$ch(L(c,h)) \sim  t^{h-\frac{c}{24}} \chi_{NS} (t) (1-t^{n})$  \end{corollary}

\subsection{Proof of the theorem}
\begin{proposition} \label{propo}
For a fixed $c$, $det_{n}$ is polynomial in $h$ of degree
\begin{displaymath} M= \sum_{\stackrel{0  < pq/2 \le n}{p \equiv q \lbrack 2 \rbrack}}d(n-pq/2) \end{displaymath}
\end{proposition}
\begin{proof}
It's clear that only the product of the diagonal entries of $M_{n}(h,c)$ gives a non-zero contribution to the highest power of $h$ (and that its coefficient is $>0$ and independent of $c$);
and that $M$ is the sum of possibles $\sum m_{i} + \sum n_{j} $ such that $\sum im_{i} + \sum jn_{j} = n$ with $i \in \NNN + \1/2 $, $j \in \NNN$, $m_{i} \in \{  0,1\} $,  $n_{j} \in \NNN $. \\
Let $m_{n}(p,q)$ be the number of such partitions of $n$, in which $p/2 $ appears exactly $q$ times; then, $M = \sum_{0  < pq/2 \le n}q.m_{n}(p,q)$.  \\
Now, if $p \equiv 0 [ 2 ] $, the number of such partitions in which $p/2$ appears $\ge q $ times is $d(n-pq/2)$; so, $m_{n}(p,q) = d(n-pq/2) - d(n-p(q+1)/2)$. \\
If $p \equiv 1 [ 2 ] $, then,  $m_{n}(p,q) = 0$ if $q>1$ and $m_{n}(p,1) = d(n-p/2) -  m_{n-p/2}(p,1)$; 
so, by induction, $m_{n}(p,1) = \sum_{q} (-1)^{q+1}d(n-pq/2) $,  \\  where $d(0) = 1 $ and $d(k) = 0$ if $k<0$.   Now:  

\begin{displaymath}  M = \sum_{\stackrel{0  < pq/2 \le n}{p \equiv 0 \lbrack 2 \rbrack} } q.m_{n}(p,q) +   \sum_{\stackrel{0  <   p/2 \le n}{p \equiv 1 \lbrack 2 \rbrack} } m_{n}(p,1)   \end{displaymath}    \begin{displaymath}  
= \sum_{\stackrel{0  < pq/2 \le n}{p \equiv 0 \lbrack 2 \rbrack} } q.(d(n-pq/2) - d(n-p(q+1)/2)) +   \sum_{\stackrel{0  < p/2 \le n}{p \equiv 1 \lbrack 2 \rbrack} } (\sum_{q} (-1)^{q+1}d(n-pq/2))   \end{displaymath}    \begin{displaymath}  
=  \sum_{\stackrel{0  < pq/2 \le n}{p \equiv 0 \lbrack 2 \rbrack} } d(n-pq/2) +  \sum_{\stackrel{0  < pq/2 \le n}{p \equiv 1 \lbrack 2 \rbrack} }(-1)^{q+1}d(n-pq/2) \end{displaymath}
Finally, the $(p,q)$ term with $q \equiv 1 [ 2 ]$ of the first sum, vanishes with the $(p',q') = (q,p)$ term of the second, so the result follows.
\end{proof}
\begin{lemma}
If $t \mapsto A(t)$ is a polynomial mapping  into $d \times d$ matrices and $dim(ker A(t_{0} )) = k $, then $(t-t_{0})^{k}$ divides $det (A(t)) $.
\end{lemma}
 \begin{proof}
 Take a basis $v_{i}$ such that $A(t_{0})v_{i} = 0$ for $i = 1 \dots k$. \\ Thus, $(t-t_{0})$ divides $A(t)v_{i}$ for $i = 1 \dots k$, and $(t-t_{0})^{k}$ divides $det(A(t))$.
\end{proof}
\begin{lemma} \label{lemma2} Consider $det_{n}(c,h)$ as polynomial in $h$ for $c$ fixed.  If $n'$ is minimal such that $det_{n'}$ vanishes at $ h = h_{0}$, then $(h-h_{0})^{d(n-n')} $ divides $det_{n}$.
\end{lemma}
\begin{proof}
Clearly $V(c, h_{0})$ admits a singular vector $s$ of level $n'$. \\ Now, $V_{n}^{s}(c,h_{0})$ is $d(n-n')$ dimensional, and is contained in $ker (M_{n}(c,h_{0}))$. \\ So, the result follows by previous lemma.
\end{proof}
\begin{lemma} \label{lemma3}
$det_{n}$ vanishes at $h_{pq}^{c}$,  for $0 < pq/2 \le n $, $p \equiv q  [ 2 ] $. 
\end{lemma}
\begin{proof} Let  $m \ge 2 $ integer,  $1 \le p \le m-1 $, $1 \le q \le m+1$,  $p \equiv q [ 2 ]$. \\
Thanks to GKO construction, we have corollary \ref{cha}:
\begin{displaymath}  ch( L(c_{m},h_{pq}^{m}) ) \le ch( M_{pq}^{m}   )  \sim t^{ - \frac{c_{m}}{24} }. \chi_{NS}(t).t^{h_{pq}^{m}  }.(1- t^{\frac{pq}{2}} - t^{\frac{p'q'}{2}})
\end{displaymath}
So, $V(c_{m},h_{pq}^{m})$ admits a singular vector at level $\le min(pq/2 , p'q'/2)$ by  corollary \ref{coro3}, and then,  $dim(ker (M_{n}(c_{m},h_{pq}^{m}))) > 0$ for $n \ge pq/2$.  Hence, $det_{n}$ vanishes at $h_{pq}^{m}$ for $m$ sufficiently large integer.
But then, $det_{n}$ vanishes at infinite many zeros of the irreducible $\varphi_{pq}$, which so, divides $det_{n}$.
\end{proof}
\paragraph{Proof of the theorem \ref{kac}} By lemma \ref{lemma2} and \ref{lemma3}, $det_{n}$ is divisible by $d_{n}(c,h) = \prod_{\stackrel{0  < pq/2 \le n}{p \equiv q \lbrack 2 \rbrack} } (h - h_{pq}^{c} )^{d(n-pq/2)}$ since the $h_{pq}^{c}$ are distincts for generic $c$. \\ Now, by proposition \ref{propo}, $det_{n}$ and $d_{n}$ have the same degree $M$, and the coefficient of $h^{M}$ is $ > 0$ and independent of $c, h$. 
So, the result follows.  $\Box$

\newpage
\section{Friedan-Qiu-Shenker unitarity criterion}
\subsection{Introduction}
Recall section 2.3 of \cite{NSOAI} for definitions of Verma module $V(c,h)$, sesquilinear form $(. , .)$ and ghost. 
The goal of this section is to give a proof of the FQS theorem for the Neveu-Schwarz algebra, in a parallel way that \cite{4d} give for the Virasoro algebra, expoiting Kac determinant formula: 
\begin{displaymath}
det_{n}(c,h) = A_{n}  \prod_{\stackrel{0  < pq/2 \le n}{p \equiv q \lbrack 2 \rbrack} } (h - h_{pq}^{c} )^{d(n-pq/2)} 
\end{displaymath} with $A_{n} > 0$ independent of $c$ and $h$.
\begin{lemma} \label{h} If $ V(c,h)$ admits no ghost then $ c,h \ge 0 $  \end{lemma}
\begin{proof}   Since $L_{n}L_{-n} \Omega =  L_{-n}L_{n} \Omega +2nh \Omega +c \frac{n(n^{2}-1)}{12}\Omega$,
  \\Ê we have $(L_{-n} \Omega,L_{-n} \Omega)  = 2nh + \frac{n(n^{2}-1)}{12}c \ge 0$.
   \\  Now, taking $n$ first equal to $1$ and then very large, we obtain the lemma.  \end{proof}
\begin{proposition} \label{propos} If $h \ge 0 $ and $c \ge 3/2$ then $V(c,h)$ admits no ghost.
\end{proposition}
Now, it suffices to classify no ghost cases for $h \ge 0 $ and $0 \le c < 3/2$. 
\begin{lemma} $ m \mapsto c_{m}$ is an inscreasing bijection from $[2,+ \infty [ $ to $ [0 , 3/2[ $. 
  \end{lemma}
 The FQS theorem gives as necessary condition exactly the same discrete series that GKO construction gives as sufficient condition (theorem \ref{suff}): 
\begin{theorem} (FQS unitary criterion) \\  \label{FQS}
Let  $h \ge 0 $ and $0 \le c < 3/2$; $V(c,h)$ admits ghost if $(c,h)$ does not belong to: 
\begin{displaymath} c = c_{m}  =  \frac{3}{2} (1 - \frac{ 8}{m(m+2) }),    \quad  h=  h_{p,q}^{m} =\frac{[(m+2)p-mq]^{2}-4}{8m(m+2)} \end{displaymath}
with integers  $m \ge 2$, \  $1 \le p \le m-1$, \   $1 \le q \le  m+1 $  and  $p \equiv q [2]$.
\end{theorem}
\begin{remark} Combining theorem \ref{suff} and lemma \ref{h}, we see that  $h_{pq}^{m} \ge 0 $
\end{remark}

\subsection{Proof of proposition \ref{propos}}  
 \begin{proof} By continuity, it suffices to treat the region $R = \{ h > 0 \ , \  c >  3/2 \} $. \\ Now, we see that $(c ,h_{pq}^{c}) \not \in R$,  
so by Kac determinant formula (theorem \ref{kac}), $det_{n}(c,h) $ is nowhere zero on $R$. So, it suffices to prove that the form is positive for one pair $(c,h) \in R$. 

If $\alpha = (a_{1}, ... , a_{r_{1}} ; b_{1}, ... , b_{r_{2}}) $, let $n(\alpha) = \sum a_{i} + \sum{b_{j}}$,  $r(\alpha) = r_{1} + r_{2}$. \\ Let  $u_{\alpha} = A_{\alpha}\Omega$, with $A_{\alpha}$ the product of $L_{-a_{i}}$ and $G_{-b_{j}}$ in the following order: \\   if $n \le m$ then $L_{-n}$ or $G_{-n}$ is before $L_{-m}$ or $G_{-m}$; example: $G_{-1/2}L_{-1}^{2}G_{-5/2} \Omega$. \\ $(u_{\alpha})$ form a basis of $V(c,h)$. \\ Now, thanks to this order, we easily prove by induction on $n(\alpha) + n(\beta)$ that:    
\begin{displaymath} (u_{\alpha}, u_{\beta}) = \left\{ \begin{array}{lr}  c_{\alpha} h^{r(\alpha)} (1 + o(1)) \ \  \textrm{with} \ c_{\alpha} > 0  &  \textrm{if} \  \alpha = \beta       \\     o(h^{(r(\alpha)+r(\beta))/2})    &  \textrm{if} \  \alpha \ne \beta        \end{array}  \right.      \end{displaymath}  So, $\forall n \in \1/2\NNN $ and $\forall u \in V_{n}(c,h)$,  $u=\sum_{n(\alpha ) = n} \lambda_{\alpha}u_{\alpha}$ and: 
\begin{displaymath} (u,u) = \sum_{\alpha , \beta}  \lambda_{\alpha} \bar{\lambda_{\beta}} (u_{\alpha},u_{\beta})  =  \sum_{\alpha} \vert   \lambda_{\alpha} \vert^{2} (u_{\alpha},u_{\alpha}) + \1/2 \sum_{\alpha \ne  \beta}  Re(\lambda_{\alpha} \bar{\lambda_{\beta}}) (u_{\alpha},u_{\beta})   > 0 \end{displaymath} for $h$ sufficiently  large and independent of  $u$. \\
Then, the form is positive for $h$ large, and so is $\forall (c,h) \in R$ by continuity.  \end{proof}

\subsection{Proof of  theorem \ref{FQS}}
\begin{definition}Let $C_{pq}$ be the  curve $h = h_{pq}^{c}$ with  $0 \ne p \equiv q [ 2 ]$.
\end{definition}
\begin{remark}
$C_{pq}$ intersects the line $c = 3/2$ at $h = \frac{(p-q)^{2}}{8} = lim_{m \to \infty} (h_{pq}^{m}) $. 
 For $ 0 \le c < \frac{3}{2}$, we see the curve as $(c_{m}, h_{pq}^{m})$ with  $m \in  [2,+ \infty [ $.  \end{remark}
 \begin{definition}  Let $\kappa = \left\{  \begin{array}{ccc}   1 & \textrm{if} & q<p+1 \\  0 & \textrm{if} & q > p+1   \end{array}  \right.$
 \end{definition}
 \begin{proposition}  \label{inter}
 When the curve  $C_{pq}$ first appears at level $n = pq/2$, if $q=1$, it intersects no other vanishing curves, else, its first intersection moving forward $c = 3/2$ is with
  $C_{q-2+\kappa, p+ \kappa}$, at $m = p+q-2+\kappa$.
  \end{proposition}
 \begin{proof}  Let  $(p',q')  \neq  (p,q)$ with $p'q' \le pq$, then the intersection points \\ $C_{pq} \cap C_{p'q'}$ are given by $[(m+2)p  - mq]^{2} = [(m+2)p'  - mq']^{2}$, with two solutions $m_{+}$ and $m_{-}$ such that $[ (p-q)  \pm  (p'-q') ] m_{\pm}  = 2(\mp p' - p)$. \\ Now, if $[ (p-q)  \pm  (p'-q') ]  = 0$ then $0 = -(p+p') \le -2 $ or $(p,q)=(p',q')$, contradiction; hence, $m_{\pm} = 2\frac{\mp p' - p}{(p-q)  \pm  (p'-q')} $ and $ \frac{1}{m_{\pm}}  = \1/2 (\frac{q \pm q' }{p \pm p'} - 1)$. \\
 If $q = 1$, we see that $\frac{q \pm q' }{p \pm p'} > 0   \Rightarrow   p'q' > pq$, contradiction.    \\   
Else, $q \neq 1$; let $(p-q) \pm (p'-q') = -2s $   with $s \in \ZZZ^{\star}$.\\ The goal is to find the biggest $m_{\pm} \in [2 , + \infty [$ among the following solutions, parametered by $s \in \ZZZ^{\star}$,  $k \in  \ZZZ $, with $p'q' \le pq$: 
\begin{itemize}
\item   $(p'_{+}, q'_{+}) = (q-s+k , p+s+k)$   and  $m_{+} = \frac{p+q+k-s}{s} $
\item   $(p'_{-}, q'_{-}) = (p+s+k , q-s+k)$   and  $m_{-} = - \frac{k-s}{s} $ 
\end{itemize} 
But, at fixed $s$ and $k$,  $m_{+} - m_{-} = \frac{p+q+2k}{s}$, and $p+q+2k = p'_{+} + p'_{-} > 0 $, so, if $s > 0$, we choose $m_{+}$, and if $s<0$, we choose $m_{-}$. \\

Let $s>0$,  $k \in \ZZZ$ and  $(p' , q') = (q-s+k, p+s+k) $. $p'q' \le pq  \Rightarrow   k < s$.\\
  The biggest m is given by $s=1$ and  $k=0$. Now, $(q-1)(p+1) > pq$ if $q > p+1$, so we take $k=-1$ in this case and so $(p', q') = (q-2+\kappa, p+\kappa)$, at $m = p+q-2+\kappa$.
 
Let $s<0$, $k \in \ZZZ$ and  $(p' , q') = (p+s+k, q-s+k) $.  $p'q' \le pq  \Rightarrow   k < -s$. \\
Now if  $- \frac{k-s}{s}  = m > p+q-2$, then $k >  -s(p+q-1) \ge -s $,  contradiction.  \end{proof}
\begin{definition}
For $q =1$, let $C'_{p1}$  be all of $C_{p1}$ for $m \ge 2 $, ie, $0 \le c \le \3/2 $, else, define $C'_{pq}$ to be the part of $C_{pq}$ for which  $m  >  p+q-2+\kappa$.
\end{definition}
$C'_{pq}$ is the open subset of  $C_{pq}$ between $c=\3/2$ and its first intersection at level $pq/2$. The first step of the proof of theorem \ref{FQS} is to eliminate all on $0 \le c \le \3/2 $, except the curves $C'_{pq}$.

\begin{definition}  Let $n \in \1/2\NNN$:  
\begin{displaymath} S_{n} = \bigcup_{\stackrel{0  < pq/2 \le n}{p \le q , \  p \equiv q \lbrack 2 \rbrack} } \{  (c,h) \  \vert \  0 \le c < \3/2 \ , \  h_{pq}^{c} \le h \le h_{qp}^{c} \textrm{ or }  h \le h_{pp}^{c}  \}
  \end{displaymath}
\end{definition}
\begin{lemma} \label{plane} $lim_{n \to \infty} S_{n}$ is all $0 \le c < \3/2 $ of the plane.
\end{lemma}
\begin{proof} $lim_{pq/2 \to \infty}(c_{p+q-2}) = 3/2$ Ê\ and \ $lim_{c \to 3/2} (h_{pq}^{c}) =  h_{pq}^{3/2} = \frac{(p-q)^{2}}{8 }$. 
\end{proof}
\begin{definition} Let $p'q'> pq $; $C_{p'q'}$ is a first intersector of $C'_{pq}$, if at level $p'q'/2$, it's the first  starting from $c=3/2$.
\end{definition}
\begin{proposition} \label{first} The first intersectors on $C'_{pq}$ are  $C_{q-1+k , p+1+k}$, $k \ge \kappa $, \\ at $m = p+q+k-1$.
\end{proposition}
\begin{proof}
We take the same structure that proof of proposition \ref{inter}. \\ $(p',q') = (q-1+k,p+1+k)$ corresponds to  $s=1$ and $k \ge \kappa  \Leftrightarrow  p'q'  >  pq$.  \\ 
Now, let $(u,v) = (q-s'+k' , p+s'+k')$ or $(p+s'+k',q-s'+k' )$, \\  if  $m'= \frac{p+q+k'-s'}{s'}$ or $- \frac{k'-s'}{s'} \ge m $  and $uv \le p'q'$, then, $k'=k$ and $s'=1$.  \\  
So, $C_{q-1+k , p+1+k}$ first intersects $C'_{pq}$.  Now, if  $m' > m-1$  and $s' \ne 1$, then, $uv > p'q'$; so, there is no other first intersector.   \end{proof}
\begin{lemma}  \label{discrete} The discrete series of theorem \ref{FQS} consists exactly of these first intersections $F_{pqk}$, on all the $C'_{pq}$.
\end{lemma}
\begin{proof} $m = p+q+k-1$ with $k \ge \kappa$, so, the set of such $m$ is $\NNN_{\ge 2}$. \\
Now, let $m \ge 2$ fixed, then, $p+q \le   m+1-\kappa $  \\
But, $h_{pq}^{m} = h_{m-p,m+2-q}^{m}$, so we obtain the discrete series:  \\
  Integers  $m \ge 2$, \  $1 \le p \le m-1$, \   $1 \le q \le  m+1 $  and  $p \equiv q [2]$.
\end{proof}
\begin{remark} \label{redondancy} We can write the series without redondancy as:  \\ $m \ge 2$, \  $1 \le p < q-1 \le m$ and $p \equiv q [2]$.  \end{remark} 
\begin{definition}
 Let  $R_{11} =  \{ 0 \le c  < 3/2 , h < 0  \}$;  \\  for $p \ne 1$, let $R_{1p} = R_{p1} $ be the open region bounded by $C'_{p1}$, $C'_{1p}$ and $C'_{p-2,1}$; \\  for $q \ne 1$, $R_{pq}$, the open region bounded by $C'_{pq}$, $C'_{p-1,q-1}$ and $C'_{q-2+\kappa,p+\kappa}$.
\end{definition}
\begin{lemma}  \label{sari}
No vanishing curves at level $n = pq/2$ intersect $R_{pq}$.
\end{lemma}
\begin{proof}
A vanishing curve which did intersect $R_{pq}$, would have to intersect its boundary. This does not happen by proposition \ref{first}.
\end{proof}
\begin{lemma} \label{Rw}
$S_{n} - S_{n-1/2} =  \bigcup_{\stackrel{ pq/2 = n}{ p \equiv q \lbrack 2 \rbrack} }  R_{pq} \cup C'_{pq}$
\end{lemma}
\begin{proof}
$S_{1/2} = R_{11} \cup C'_{11}$, \  $C_{pq} - C'_{pq} \subset S_{n-1/2}$ and  lemma \ref{sari}. 
   
\end{proof}
\begin{lemma} \label{elim}
All $S_{n}$ is eliminated,  except  $C'_{pq}$, $pq/2 \le n$.
\end{lemma}
\begin{proof}
By previous lemma, $S_{n} = \bigcup_{\stackrel{ pq/2 \le n}{ p \equiv q \lbrack 2 \rbrack} }  R_{pq} \cup C'_{pq}$. \\
Now, we see that, for $p \ne q$, $R_{pq}$  is between $C_{pq}$ and $C_{qp}$; $R_{pp}$ is under $C_{pp}$, and for $p'q' \le pq$ with  $(p',q') \ne (p,q)$, $R_{pq}$ is necessarily over  $C_{p'q'}$ and $C_{q'p'}$, or under them. So (recall section \ref{pre}), $\varphi_{pq}(c,h) < 0$ and  $\varphi_{p'q'}(c,h) > 0$ on $R_{pq}$, and $d(0) = 1$; then, $det_{pq/2}(c,h) < 0$ and $V(c,h)$ admits ghosts on $R_{pq}$.
\end{proof}
Now, given lemma \ref{plane} and \ref{elim}, we have to eliminate the intervals on  $C'_{pq}$,  between the points of the discrete series.
\begin{definition}
Let $I_{pqk}$ be the open subset of $C'_{pq}$ between $F_{p,q,k-1}$ and $F_{p,q,k}$ for $k >  \kappa$; and  $I_{pq\kappa}$,  beyond $F_{pq\kappa}$.
\end{definition}
\begin{lemma} \label{kra}  $C'_{pq} = \bigcup_{k \ge k_{0}} I_{pqk} \cup  F_{pqk}$.
\end{lemma}
The goal is to eliminate the open subset $I_{pqk}$, $k \ge \kappa$. \\ Recall that when $C_{p'q'} = C_{q-1+k , p+1+k}$ first appears at level $n' = p'q'/2$, there is a  ghost on $R_{p'q'}$; we will show that this ghost continue to exist on $I_{pqk}$.
\begin{proposition}
 At level $n' = p'q'/2$, the first $k-\kappa + 1$ successives intersections on  $C_{p'q'}$ are with $C'_{p+k-j,q+k-j} $  ($\kappa \le j \le k $) at its first intersection $F_{p+k-j,q+k-j,j}$, with $m = p+q+2k-j-1$
\end{proposition}
\begin{proof}
Let $ (p'', q'') = (q'-s+k' , p'+s+k') $. \\  If  $p''q'' \le p'q' $  and,  $\frac{p'+q'+k'-s'}{s'} $ or $-\frac{k'+s' }{s' }  \ge m= p+q+k-1$, (ie, with $j = k$), \\  then $s'=1$;
  now, by proposition \ref{inter},  the first is with $j=\kappa$.
\end{proof} 

\begin{lemma}
Let $M_{t}$ be an $d$-dimensional polynomial matrix with $det(M_{t})$ vanishing to first order at $t=0$; then, the null space is $1$-dimensional.
\end{lemma}
\begin{proof}
Let $\alpha_{1}(t), ... , \alpha_{d}(t) $ be the eigenvalues of $M_{t}$; they are analytic in $t$. Now, $det(M_{t}) = \prod \alpha_{i}(t) =  \prod ( \alpha_{i}^{0} + \alpha_{i}^{1}t + ...)$, vanishing to first order at $t=0$, so, there exists a unique $i$ such that $\alpha_{i}^{0} = 0$, and $dim ker M_{0} = 1$.
\end{proof}  
\begin{corollary}
Let $(c,h) \in C_{pq}$,  not on an intersection at level $ pq/2$, then, the null space of $V_{pq/2}(c,h)$ is $1$-dimensional.
\end{corollary}
\begin{lemma} \label{biz} Let $(c,h) = F_{pqk}$, then,  $det_{(p'q' - pq)/2}(c, h+pq/2) \ne 0$. 
\end{lemma}
\begin{proof} If this determinant were zero, then $(c, h+pq)$ would be on a vanishing curve $C_{uv}$ of level  $\le \1/2(p'q' - pq)$:   $h_{pq}^{m} + pq/2 = h_{uv}^{m}$ and $uv \le p'q' -pq$. \\
Then,  we find  $(u,v)$ or $(v,u) = (ms'-p, (m+2)s'+q)$, with $s'\in \ZZZ^{\star}$. \\  
So now, $uv \le p'q' -pq $ is equivalent to $((1+s')m -p)((1-s')(m+2)-q) \ge 0$, but $1 \le p < m $ and $1 \le q < m+2$, so, $s'=0$, contradiction.
\end{proof}
To read the followings proposition and its proof,  recall section \ref{singulars}. \\ It's strictly parallel that in \cite{4d} for the Virasoro algebra.
\begin{proposition}   \label{singe}
For $j=\kappa, ... ,k$ there is an open neighborhood $U_{p'q'j}$  of $F_{p+k-j,q+k-j,j} = F_{q'-1-j,p'+1-j,j}$ and a nowhere zero analytic function  $v_{j}(c,h)$ defined on $U_{p'q'j}$ with values in $V_{n'}(c,h)$, with $n'=p'q'/2$, such that:  \\ $\begin{array}{c}   \end{array}$  \hspace{3,5cm} $v_{j}(c,h) \in K_{n}(c,h)  \Leftrightarrow  (c,h) \in C_{p'q'}$
\end{proposition}
\begin{proof} 
Write $p'' = p+k-j$, $q'' = q+k-j$ and $n'' = p''q''/2 < n'$. \\ Let $U = U_{p'q'j}$ be a neighborhood  of $F_{p+k-j,q+k-j,j}$, small enough that it intersects no vanishing curves but $C_{p'q'}$ and $C_{p''q''}$ at level $n'$. Choose coordinates $(x,y)$ in $U$, real analytic in $(c,h)$, such that $C_{p''q''}$ is given by $x=0$ and $C_{p'q'}$ by $y=0$. This is possible because the intersection is transversal. At level $n''$,  $x=0$ is the only vanishing curve in $U$. $ K_{n''}(0,y)$ is  one dimensional  and form a line bundle over the vanishing curve $x=0$ near $y=0$. Let $v''_{j}(0,y)$ be a nowhere zero analytic section of this line bundle, and let $v''_{j}(x,y)$ be an analytic function on $U$ with values in $V_{n''}(x,y)$, which extends this section. Let $V''(x,y) = V^{v''_{j}}_{n'}(x,y)$ of dimension $d(n'-n'')$. For $y \ne 0$, the order of vanishing of $det_{n'}(x,y)$ at $x=0$ is also $d(n'-n'')$. Therefore, for $y \ne 0$, $V''(0,y) = K_{n'}(0,y)$. Let $V'(x,y)$ such that $V_{n'} = V'' \oplus V'$ and we write:
\begin{displaymath}  M_{n'}(x,y) = \left( \begin{array}{cc} xQ(x,y)  &  x R(x,y)  \\  xR(x,y)^{t}  &   S(x,y) \end{array} \right)  \end{displaymath} with $Q$, $S$  symmetric and $3$ blocks divisible by $x$ because  $V''(0,y) \subset  K_{n'}(0,y)$. \\

 The key point now, is that $Q(0,0)$ is non-degenerate. \\ To see this, first note that $v''_{j}(0,y)$ is singular, $M_{n'}(0,y)v''_{j}(0,y) = 0$ and $L_{0}v''_{j}(0,y) = (h+p''q''/2)v''_{j}(0,y)$; recall that $(0,y)=(c,h) \in C_{p''q''}$. \\ Now, since all is analytic, $\forall \alpha, \beta \in V''(x,y) $: \\
  \begin{displaymath} (\alpha, \beta) = (A.v''_{j}(x,y) , B.v''_{j}(x,y) ) = ([B^{\star},A]v'',v'') + (B^{\star}v'',A^{\star}v'') \end{displaymath} 
 \begin{displaymath}  = ([B^{\star},A]\widetilde{\Omega},\widetilde{\Omega}) (v'',v'') + o(x) ÊÊ= cte.x(A.\widetilde{\Omega} , B.\widetilde{\Omega} ) + o(x) , \end{displaymath}  with $\widetilde{\Omega}$ the cyclic vector of $V(c,h+p''q''/2)$;  so:  \begin{displaymath} Q(x,y) = M_{(p'q'-p''q'')/2}(c,h+p''q''/2) + x.M'(x,y).  \end{displaymath} Since $(0,0)=F_{p''q''j}$,  lemma \ref{biz} gives $det(Q(0,0)) \ne 0$;  so, $Q(x,y)$ is non-degenerate on all $U$ (we can replace $U$ by a small neighborhood of $(0,0)$). \\  Let $W = \left( \begin{array}{cc} 1  &  -Q^{-1}  \\  0  &   1 \end{array} \right)$ and make the change of basis: 
   \begin{center}$ M_{n'} \mapsto W^{t} M_{n'} W =  \left( \begin{array}{cc} xQ(x,y)  & 0  \\  0  &   T(x,y) \end{array} \right)$ \end{center}
   Let $V'''(x,y)$ be the new complement of $V''(x,y)$, on which $T(x,y)$ defined the inner product. The order of vanishing argument implies that $det(T(x,y))$ is non-zero for $y \ne 0$ and vanishes to first order at $y = 0$. The one dimensional null space of $T(x,0)$ is $K_{n'}(x,0)$ for $x \ne 0$. At $x=y=0$, the one dimensional null space of $T(0,0)$ and $V''(0,0)$, span the d(n'-n'')+1 dimensional $K_{n'}(0,0)$. By the same argument which gave $v''_{j}(x,y)$, we can choose a nowhere zero analytic function $v_{j}(x,y)$ on $U$, with values in $V'''(x,y)$ such that $v_{j}(x,0)$ is in the null space of $T(x,0)$ and therefore in $K_{n'}(x,0)$. Since $T(x,y)$ is non-degenerate for $y \ne 0$, $v_{j}(x,0)$ is not in $K_{n'}(x,y)$ if $y \ne 0$  \end{proof}
\begin{definition}
Let  $ J_{p'q'j}$, $\kappa < j \le k$, be the open interval on $C_{p'q'}$ between $F_{p+k-j,q+k-j,j}$ and $F_{p+k-j-1,q+k-j-1,j}$, and let $J_{p'q'\kappa}$ be the open interval on  $C_{p'q'}$ lying between $c=3/2$ and $F_{p+k-\kappa,q+k-\kappa,\kappa}$. 
\end{definition}
\begin{definition}
Let $ W_{p'q'j}$, $\kappa \le j \le k$ be a neighborhood of a point of $ J_{p'q'j}$, which intersects no other vanishing curves on level $n'$, such that: : \\ $ J_{p'q'j} \subset U_{p'q'j-1} \cup W_{p'q'j} \cup U_{p'q'j}$ if $j >  \kappa$, and $\varnothing \ne U_{p'q'\kappa} \cap W_{p'q'\kappa}  \subset R_{p'q'}$
\end{definition}

\begin{lemma}  \label{le}
For each $j$, $\kappa \le j \le k$, there is a nowhere zero analytic function $w_{j}(c,h)$ on $W_{p'q'j}$ with values in $V_{n'}(c,h)$, such that $w_{j}(c,h)$ is in $K_{n'}(c,h)$ if and only if $(c,h)$ is on $ J_{p'q'j}$, and:  \begin{displaymath} w_{j} = \left\{  \begin{array}{l}  f_{j}v_{j}  \ \textrm{on} \  W_{p'q'j} \cap U_{p'q'j}  \\     g_{j}v_{j-1}  \ \textrm{on} \  W_{p'q'j} \cap U_{p'q'j-1}   \  (j \ne \kappa)       \end{array} \right. \end{displaymath}  where $ f_{j}$, $g_{j}$ are nonzero function.
  \end{lemma}
\begin{proof}  
$K_{n'}(c,h)$ is trivial on  $W_{p'q'j}$, except on $J_{p'q'j}$, where $dim(K_{n'}) = 1$.
 \end{proof}
\begin{lemma}  \label{amm}
  $I_{pqk}$ is eliminated on level $n' = (q-1+k)(p+1+k)/2$.
 \end{lemma}
\begin{proof}  
By proposition \ref{propos}, $M_{n'}(c,h)$ is positive on $h \ge 0$, $c \ge 3/2 $. \\ Now, at level $n'$, we can go from this sector to $W_{p'q'\kappa}$ without crossing a vanishing curve, so, $(w_{\kappa}, w_{\kappa}) > 0$ before crossing $C_{p'q'}$.  But it vanishes to first order on $C_{p'q'}$, so, after crossing it, $w_{\kappa}$ becomes a ghost.  Now, by lemma \ref{le} and induction, so is for $v_{\kappa}$, $w_{\kappa + 1}$, $v_{\kappa +1}$, ... up to $v_{k}(c,h) \in I_{pqk} \cap U_{p'q'k}$. \\ Finally, $v_{k}(c,h)$  continues to be a ghost on all $I_{pqk}$, because $I_{pqk}$  cross no other vanishing  curve on level $n'$.
 \end{proof}
Lemmas \ref{plane}, \ref{elim}, \ref{kra} and \ref{amm} imply theorem \ref{FQS} and  theorem \ref{satya}.  

\newpage
\section{Wassermann's argument}     \label{wasserone}
We need to recall sections \ref{chari} and  \ref{singulars}; by lemma \ref{discrete} the discrete series 
are the intersections of $C'_{pq}$ and $C_{p'q'}$ at $m = p+q+k-1$, $k \ge \kappa$, with
$(p',q') = (q-1+k ,  p+1+k) = (m-p , m+2-q)$, ie, $h_{pq}^{m} = h_{m-p, m+2-q}^{m}$.  \\ 
Let $M = max(pq/2,p'q'/2)$. This section will prove theorem \ref{sati}, thanks to an argument that A. Wassermann uses for the Virasoso case in \cite{1}.

\begin{lemma} \label{quartz} At level $\le M$, we find only two singular vectors $s$ and $s'$ \\ at level $pq/2$ and $p'q'/2$.
 \end{lemma}
\begin{proof} We can suppose $p'q' > pq$; by proof of proposition \ref{singe}: 
\begin{displaymath}
K_{n}(c_{m},h_{pq}^{m}) =  \left\{  \begin{array}{lcl}    \{ 0 \}  &  \textrm{if}   & n < pq/2  \\ \CCC s &  \textrm{if} & n = pq/2  \\ V_{n}^{s}(c_{m},h_{pq}^{m})  &   \textrm{if} & pq/2 \le n < p'q'/2 \\ V_{n}^{s}(c_{m},h_{pq}^{m}) \oplus \CCC s' &  \textrm{if} & n = p'q'/2   \end{array}   \right.
\end{displaymath}
Then, by proposition \ref{cheval}, the result follows.
\end{proof}
\begin{corollary}  \label{dog}
$ch(L(c_{m} , h_{pq}^{m})) \sim  \chi_{NS}(t) . t^{h_{pq}^{m} - c_{m}/24} (1 - t^{pq/2} - t^{p'q'/2})$
\end{corollary}
\begin{proof}
By  section \ref{singulars} and lemma \ref{quartz}.
\end{proof}
\begin{lemma} \label{chine}
$h_{pq}^{m} + M > m^{2}/8$
\end{lemma}
\begin{proof}
$h_{pq}^{m} + M = max ( \gamma_{-p,q}^{m}(0) , \gamma_{-p,q}^{m}(-1))$. \\
 $\gamma_{-p,q}^{m}(0) = \frac{x^{2}-4}{8m(m+2)}$, $\gamma_{-p,q}^{m}(-1) = \frac{(x-2m(m+2))^{2}-4}{8m(m+2)}$, with  $x = (m+2)p+mq$. \\
If $\gamma_{-p,q}^{m}(0) > m^{2}/8$, it's ok. \\
Else, $ \frac{x^{2}-4}{8m(m+2)} \le  m^{2}/8   \Leftrightarrow   x^{2} \le m^{4} + 2m^{2} + 4 < (m+1)^{4}$ \\ 
So, $\gamma_{-p,q}^{m}(-1) = \frac{[2m(m+2)-x]^{2}-4}{8m(m+2)} >  \frac{[2m(m+2)-(m+1)^{2}]^{2}-4}{8m(m+2)} \ge \frac{m^{4} + 2m^{3}}{8m(m+2)} = m^{2}/8 $.
\end{proof}
\begin{theorem}
The multiplicity space $M_{pq}^{m}$ is exactly $L(c_{m},h_{p,q}^{m})$.
\end{theorem}
\begin{proof}
By corollary \ref{coro}, $L(c_{m},h_{p,q}^{m})$ is a $ \Vir_{1/2}$-submodule of $M_{pq}^{m} $; 
if $M_{pq}^{m}$ admits another irreducible submodule (of central charge $c_{m}$), then, by theorem \ref{FQS}, it is on the discrete series, of the form $L(c_{m},h_{rs}^{m})$. 
Now, by lemma \ref{lem} and corollary \ref{dog}: 
 $ch(M_{pq}^{m})-ch(L(c_{m} , h_{pq}^{m})) =   \chi_{NS}(t) . t^{- c_{m}/24} o( t^{h_{pq}^{m} + M})$.
So we need $h_{rs}^{m} > M+h_{pq}^{m}$;
but, $h_{rs}^{m} = \frac{[(m+2)r-ms]^{2}-4}{8m(m+2)} \le \frac{(m^{2} - 2)^{2}-4}{8m(m+2)} = \frac{ m(m-2)}{8 }$. 
So, by lemma  \ref{chine}, $\frac{ m^{2}}{8 } < M+h_{pq}^{m} < h_{rs}^{m} \le  \frac{ m(m-2)}{8 }$, contradiction.
\end{proof}

\begin{theorem} The characters of the discrete series are: 
\begin{displaymath} ch(L( c_{m} , h_{pq}^{m}))(t) =\chi_{NS}(t).\Gamma^{m}_{pq}(t).t^{-c_{m}/24} \ \  \textrm{with}    \end{displaymath}
\begin{displaymath}\chi_{NS}(t) = \prod_{n \in \NNN^{\star}}\frac{1+t^{n-1/2}}{1-t^{n}},   \ \ \  \Gamma^{m}_{pq}(t) = \sum_{n \in \ZZZ}(t^{\gamma^{m}_{pq}(n)}-t^{\gamma^{m}_{-pq}(n)} )\ \  \textrm{and} \end{displaymath}
\begin{displaymath} \gamma^{m}_{pq}(n) =  \frac{[2m(m+2)n-(m+2)p+mq]^{2}-4}{8m(m+2)} \end{displaymath}   \end{theorem}
\begin{proof}
$ ch(L(c_{m} , h_{pq}^{m})) =ch(M_{pq}^{m})$, the result follows by corollary \ref{space}.
\end{proof}
\begin{remark} (Tensor product decomposition) 
\begin{displaymath}  \F_{NS}^{\gg}  \otimes L(j,\ell)= \bigoplus_{\stackrel{1  \le q \le m+1}{p \equiv q \lbrack 2 \rbrack} } L(c_{m} , h_{pq}^{m}) \otimes  L(k,\ell + 2)  \end{displaymath}
with $p=2j+1$, $q = 2k+1$,  $m=\ell + 2$ and $\gg = \sl_{2}$.   \end{remark}
We then recover a result due to Frenkel in \cite{ff}: 
\begin{corollary} \label{realferm} $ \F_{NS}^{\gg} = L(0,2)  \oplus L(1,2) $  as $L\gg$-module.   \end{corollary}
\begin{proof} It suffices to take $j = \ell = 0$, and to see that $c_{2} = h_{11}^{2} = h_{13}^{2} = 0$.   \end{proof}

\begin{corollary}(Duality)Let $H$ be an irreducible     positive energy represemtation of the loop superalgebra $\widehat{\gg}\oplus \widehat{\gg} $, let $A $ be the operator algebra generated by the  modes of the coset operators $L_{n}$ and $G_{r}$, let $B$ be the operator algebra generated by the  modes of the  diagonal loop superalgebra $ \widehat{\gg}$. Then, $A $ and $B $ are each other algebraic graded commutant (see \cite{1}).  \end{corollary}

\begin{definition} (Vertex algebra supercommutant or centralizer algebra ) \\ Let $V$ be a vertex superalgebra and $W$ a vertex sub-superalgebra, then, the vertex algebra supercommutant of $W$ is the vertex superalgebra corresponding to the vectors $v \in V$ such that the modes of the corresponding field supercommute with the modes of fields for vectors of $W$ (see \cite{6}).   \end{definition}

\begin{corollary} (Vertex superalgebra duality) In the vertex superalgebra generated by $\widehat{\gg}\oplus \widehat{\gg} $, the vertex superalgebras generated by the Neveu-Schwarz coset and the diagonal loop superalgebra, are each others supercommutants.      \end{corollary}

  \end{document}